\documentclass[11pt]{article}
\usepackage{amsmath}
\usepackage{amssymb}
\usepackage{amsthm}
\usepackage[usenames]{color}
\usepackage[colorlinks=true,linkcolor=webgreen,filecolor=webbrown,
citecolor=webgreen]{hyperref}

\definecolor{webgreen}{rgb}{0,.5,0}
\definecolor{webbrown}{rgb}{.6,0,0}

\def\C{{\Bbb C}}
\def\N{{\Bbb N}}
\def\Z{{\Bbb Z}}
\def\RE{\operatorname{Re}}
\def\id{\operatorname{id}}
\def\lcm{\operatorname{lcm}}
\def\a{{\bf a}}
\newcommand{\DOT}{\text{\rm\Huge{.}}}

\newtheorem{theorem}{Theorem}
\newtheorem{cor}{Corollary}

\newtheorem{lemma}{Lemma}

\begin{document}

\title{\bf Sums of products of Ramanujan sums}
\author{L\'aszl\'o T\'oth \thanks{The author gratefully acknowledges support from the
Austrian Science Fund (FWF) under the project Nr. P20847-N18.}\\ \\
Department of Mathematics, University of P\'ecs \\ Ifj\'us\'ag u. 6, H-7624 P\'ecs, Hungary \\
and
\\ Institute of Mathematics, Department of Integrative Biology \\
Universit\"at f\"ur Bodenkultur, Gregor Mendel-Stra{\ss}e 33, A-1180
Wien, Austria \\  E-mail: ltoth@gamma.ttk.pte.hu}
\date{}
\maketitle

\centerline{Annali dell' Universita di Ferrara, Volume 58, Number 1 (2012), 183-197}

\begin{abstract} The Ramanujan sum $c_n(k)$ is defined as the sum
of $k$-th powers of the primitive $n$-th roots of unity. We
investigate arithmetic functions of $r$ variables defined as certain
sums of the products $c_{m_1}(g_1(k))\cdots c_{m_r}(g_r(k))$, where
$g_1,\ldots, g_r$ are polynomials with integer coefficients. A
modified orthogonality relation of the Ramanujan sums is also
derived.
\end{abstract}



\section{Introduction}

Let $c_n(k)$ denote, as usual, the Ramanujan sum defined as the sum
of $k$-th powers of the primitive $n$-th roots of unity ($k\in \Z$,
$n\in \N:=\{1,2,\ldots\}$), i.e.,
\begin{equation}  \label{Ramanujan_sum}
c_n(k):= \sum_{\substack{j=1 \\ \gcd(j,n)=1}}^n \exp(2\pi i jk/n),
\end{equation}
which can be expressed as
\begin{equation} \label{Ramanujan_repr}
c_n(k)=\sum_{d \mid \gcd(k,n)} d\mu(n/d),
\end{equation}
where $\mu$ is the M\"obius function. It follows from
\eqref{Ramanujan_repr} that all values of $c_n(k)$ are integers. The
sums $c_n(k)$ enjoy the properties
\begin{equation} \label{one_variable_E}
\frac1{n}\sum_{k=1}^{n} c_n(k)=
\begin{cases} 1, \  & n=1, \\  0, & \text{otherwise}, \end{cases}
\end{equation}
\begin{equation} \label{two_variable_E}
\frac1{\lcm(\ell,n)} \sum_{k=1}^{\lcm(\ell,n)} c_{\ell}(k)c_n(k)=
\begin{cases} \phi(n), \  & \ell=n, \\  0, & \text{otherwise}, \end{cases}
\end{equation}
where $\phi$ is Euler's totient function. Ramanujan \cite{Ram1918}
derived pointwise convergent series representations of arithmetic
functions $f\colon \N \to \C$ with respect to the sums
\eqref{Ramanujan_sum}, of the form
\begin{equation*}
f(k)= \sum_{n=1}^{\infty} a_f(n) c_n(k) \qquad (k\in \N)
\end{equation*}
with certain coefficients $a_f(n)$, allowed by the orthogonality
formula \eqref{two_variable_E}. See also the book of Schwarz and
Spilker \cite {SchSpi1994} and the recent survey paper of Lucht
\cite{Luc2010}.

Ramanujan sums play also an important role in the proof of
Vinogradov's theorem concerning the number of representations of an
odd integer as the sum of three primes. See, e.g., Nathanson
\cite[Ch.\ 8]{Nat1994}.

Let $m_1,\ldots,m_r \in \N$ ($r\in \N$) and
$m:=\lcm(m_1,\ldots,m_r)$, throughout the paper. The function
\begin{equation} \label{def_func_E}
E(m_1,\ldots,m_r):= \frac1{m}\sum_{k=1}^m c_{m_1}(k)\cdots
c_{m_r}(k)
\end{equation}
has combinatorial and topological applications, and was investigated
in the recent papers of Liskovets \cite{Lis2010} and of the author
\cite{Tot2011}. Note that all values of $E(m_1,\ldots,m_r)$ are
nonnegative integers. Furthermore, the function $E$ is
multiplicative as a function of several variables (see Section
\ref{section_2} for the definition of this concept). Remark that in
the case of one and two variables \eqref{def_func_E} reduces to
\eqref{one_variable_E} and \eqref{two_variable_E}, respectively.

Recall the following identity, which is due to Cohen \cite[Cor.\
7.2]{Coh1959},
\begin{equation} \label{Cohen_id} \sum_{\substack{k=1\\
\gcd(k,n)=1}}^n c_n(k-a)= \mu(n)c_n(a) \qquad (n\in \N, a\in \Z).
\end{equation}

These suggest to consider the following generalizations of
\eqref{def_func_E} and \eqref{Cohen_id}. Let $G=(g_1,\ldots,g_r)$ be
a system of polynomials with integer coefficients and let define
\begin{equation} \label{def_func_E_G}
E_G(m_1,\ldots,m_r):= \frac1{m} \sum_{k=1}^m c_{m_1}(g_1(k))\cdots
c_{m_r}(g_r(k)),
\end{equation}
\begin{equation} \label{def_func_R_G}
R_G(m_1,\ldots,m_r):= \sum_{\substack{k=1\\\gcd(k,m)=1}}^m
c_{m_1}(g_1(k))\cdots c_{m_r}(g_r(k)),
\end{equation}
where we can assume that $m_i>1$ ($1\le i\le r$), since $c_1(k)=1$
for any $k\in \Z$.

If $g_1(x)=\ldots =g_r(x)=x$, then \eqref{def_func_E_G} reduces to
the function \eqref{def_func_E}. In the one variable case, i.e., $r=1$
and selecting the linear polynomial $g_1(x)=x-a$, \eqref{def_func_R_G} gives the sum in
\eqref{Cohen_id}.

Consider also the sum
\begin{equation} \label{def_T_a}
T_a(m_1,\ldots,m_r):= \sum_{\substack{k_1,\ldots,k_{r-1},\ell \text{ (mod $m$)} \\
\gcd(\ell,m)=1}} c_{m_1}(k_1)\cdots
c_{m_{r-1}}(k_{r-1})c_{m_r}(k_1+\ldots +k_{r-1}+\ell-a),
\end{equation}
where $a\in \Z$, representing another arithmetic function of $r$
variables, which reduces to the sum of \eqref{Cohen_id} for $r=1$.

In this paper we evaluate the functions $E_G$, $R_G$, $T_a$, show
that they are all multiplicative and investigate some special cases, including
the function
\begin{equation} \label{def_func_R}
R(m_1,\ldots,m_r):= \sum_{\substack{k=1 \\
\gcd(k,m)=1}}^m c_{m_1}(k-1)\cdots c_{m_r}(k-1),
\end{equation}
obtained from \eqref{def_func_R_G} by selecting $g_1(x)=\ldots
=g_r(x)=x-1$.

In Section \ref{section_T} we derive a modified orthogonality
relation of the Ramanujan sums by evaluating the function $T_a$.

Similar results as those of the present paper, with $\gcd(g(k),m_i)$
instead of $c_{m_i}(g(k))$ ($1\le i\le r$) in \eqref{def_func_E_G}
and \eqref{def_func_R_G}, generalizing Menon's identity, are
presented in \cite{Tot2011_2}.

\section{Preliminaries} \label{section_2}

We recall that an arithmetic function of $r$ variables is a function
$f:\N^r \to \C$, denoted by $f\in {\cal F}_r$. The function $f$ is
called multiplicative if it is nonzero and
\begin{equation*}
f(m_1n_1,\ldots,m_rn_r)= f(m_1,\ldots,m_r) f(n_1,\ldots,n_r)
\end{equation*}
holds for any $m_1,\ldots,m_r,n_1,\ldots,n_r\in \N$ such that
$\gcd(m_1\cdots m_r,n_1\cdots n_r)=1$.

If $f$ is multiplicative, then it is determined by the values
$f(p^{a_1},\ldots,p^{a_r})$, where $p$ is prime and
$a_1,\ldots,a_r\in \N \cup \{0\}$. More exactly, $f(1,\ldots,1)=1$
and for any $m_1,\ldots,m_r\in \N$,
\begin{equation*}
f(m_1,\ldots,m_r)= \prod_p f(p^{e_p(m_1)}, \ldots,p^{e_p(m_r)}),
\end{equation*}
where $m_i=\prod_p p^{e_p(m_i)}$ are the prime power factorizations
of $m_i$ ($1\le i\le r$), the products are over the primes $p$ and
all but a finite number of the exponents $e_p(m_i)$ are zero.

If $r=1$, i.e., in the case of functions of a single variable we
reobtain the familiar notion of multiplicativity.

For example, the functions $(m_1,\ldots,m_r) \mapsto
\gcd(m_1,\ldots,m_r)$ and $(m_1,\ldots,m_r) \mapsto \linebreak
\lcm(m_1,\ldots,m_r)$ are multiplicative for every $r\in \N$.

The product and the quotient of (nonvanishing) multiplicative
functions are multiplicative. Let $h\in {\cal F}_1$ and $f\in {\cal
F}_r$ be multiplicative functions. Then the functions
$(m_1,\ldots,m_r) \mapsto h(m_1)\cdots h(m_r)$ and $(m_1,\ldots,m_r)
\mapsto h(f(m_1,\ldots,m_r))$ are multiplicative. In particular, \\
$(m_1,\ldots,m_r) \mapsto h(\gcd(m_1,\ldots,m_r))$ and
$(m_1,\ldots,m_r) \mapsto h(\lcm(m_1,\ldots,m_r))$ are
multiplicative.

If $f,g\in {\cal F}_r$, then their convolution is defined by
\begin{equation} \label{convo_functions}
(f*g)(m_1,\ldots,m_r): = \sum_{d_1\mid m_1, \ldots, d_r\mid m_r}
f(d_1,\ldots,d_r) g(m_1/d_1, \ldots, m_r/d_r).
\end{equation}

The convolution \eqref{convo_functions} preserves the
multiplicativity of functions. For these and further properties of
arithmetic functions of several variables we refer to
\cite{Siv1989,Vai1931}.

Let $G=(g_1,\ldots,g_r)$ be a system of polynomials with integer
coefficients and consider the simultaneous congruences
\begin{equation} \label{sim_cong_m}
g_1(x)\equiv 0 \text{ (mod $m_1$)}, \ldots, g_r(x)\equiv 0 \text{
(mod $m_r$)}.
\end{equation}

Let $N_G(m_1,\ldots,m_r)$ denote the number of solutions $x$ (mod
$\lcm(m_1,\ldots,m_r)$) of \eqref{sim_cong_m}. Furthermore, let
$\eta_G(m_1,\ldots,m_r)$ denote the number of solutions $x$ (mod
$\lcm(m_1,\ldots,m_r)$) of  \eqref{sim_cong_m} such that
$\gcd(x,m_1)=1$, ..., $\gcd(x,m_r)=1$.

\begin{lemma} {\rm (\cite[Lemma 1]{Tot2011_2})} \label{lemma_sim_cong}
For every system $G=(g_1,\ldots,g_r)$ of polynomials with integer
coefficients the functions $(m_1,\ldots,m_r) \mapsto
N_G(m_1,\ldots,m_r)$ and $(m_1,\ldots,m_r) \mapsto
\eta_G(m_1,\ldots,m_r)$ are multiplicative.
\end{lemma}

We also need the following lemmas.

\begin{lemma} {\rm (see e.g., \cite[Th.\ 10.4]{Ore1988})} \label{lemma_cong}
For every $d_1,\ldots,d_r\in \N$ and $a_1,\ldots,a_r \in \Z$ the
simultaneous congruences
\begin{equation*}
x \equiv a_1 \ (\text{mod $d_1$}), ..., x \equiv a_r \ (\text{mod
$d_r$})
\end{equation*}
admit a solution in $x$ if and only if $\gcd(d_i,d_j)\mid a_i-a_j$
($1\le i,j\le r$), and in this case there is a unique solution $x$
(mod $\lcm(d_1,\ldots,d_r)$).
\end{lemma}

\begin{lemma} {\rm (\cite[Th.\ 5.32]{Apo1976})} \label{lemma_phi}
Let $n,d,x\in \N$ such that $d\mid n$, $1\le x\le d$,
$\gcd(x,d)=1$. Then
\begin{equation*}
\# \{k\in \N: 1\le k \le n, k\equiv x \ \text{\rm (mod $d$)},
\gcd(k,n)=1 \}=\phi(n)/\phi(d).
\end{equation*}
\end{lemma}

\begin{lemma} \label{Brauer_Rademach} {\rm (Brauer-Rademacher identity, cf.
\cite[Ch.\ 2]{McC1986})} For every $n,k\in \N$,
\begin{equation} \label{Brauer_Rademacher}
\sum_{d\mid n, \gcd(d,k)=1}
\frac{d\mu(n/d)}{\phi(d)}=\frac{\mu(n)c_n(k)}{\phi(n)}.
\end{equation}
\end{lemma}

Finally we present some concepts and results concerning $s$-even
functions, to be applied in Section \ref{section_T} involving the
function $T_a$. Usually the term $r$-even function, or even function
(mod $r$) is used in the literature, see \cite{Coh1959,McC1986,
SchSpi1994,Siv1989,TotHau2010}, but we replace here $r$ by $s$.

A function $f\in {\cal F}_1$ is said to be an $s$-even function if
$f(\gcd(n,s))=f(n)$ for all $n\in \N$, where $s\in \N$ is fixed.
Then $f$ is $s$-periodic, i.e., $f(n+s)=f(n)$ for every $n\in \N$
and this periodicity extends $f$ to a function defined on $\Z$.

If $f$ is $s$-even, then it has a (Ramanujan-)Fourier expansion of
the form
\begin{equation*}
f(n)= \sum_{d\mid s} \alpha_f(d) c_d(n) \qquad (n\in \N),
\end{equation*}
where the (Ramanujan-)Fourier coefficients $\alpha_f(d)$ are
uniquely determined and given by
\begin{equation*}
\alpha_f(d)= \frac1{s} \sum_{e\mid s} f(e)c_{s/e}(s/d).
\end{equation*}

The Cauchy convolution of the $s$-even functions $f$ and $g$ is
given by
\begin{equation} \label{Cauchy_convo}
(f\otimes g)(n):= \sum_{ k \text{ (mod $s$)}} f(k)g(n-k)  \qquad
(n\in \N),
\end{equation}
$f\otimes g$ is again $s$-even and $\alpha_{f\otimes g}(d)=s
\alpha_f(d) \alpha_g(d)$ for every $d\mid s$.

\begin{lemma} {\rm (\cite[Th.\ 6]{Coh1959})} \label{lemma_4}
If $f$ is $s$-even, then for every $a\in \Z$,
\begin{equation} \label{Cohen_id_general}
\sum_{\substack{k=1\\ \gcd(k,s)=1}}^s f(a-k)= \phi(s) \sum_{d\mid s}
\frac{\alpha_f(d) \mu(d)}{\phi(d)} c_d(a),
\end{equation}
\end{lemma}

For $f(k)=c_s(k)$ ($k\in \N$), which is $s$-even, formula
\eqref{Cohen_id_general} reduces to \eqref{Cohen_id}, with $n$
replaced by $s$.  See also \cite[Ch.\ 2]{McC1986}.

\section{The function $E_G$}

Here we consider the function $E_G$ defined by \eqref{def_func_E_G}.

\begin{theorem} \label{theorem_E} If $G$ is an arbitrary system of polynomials with
integer coefficients, then for any $m_1,\ldots,m_r\in \N$,
\begin{equation} \label{repr_E}
E_G(m_1,\ldots,m_r) =  \sum_{d_1\mid m_1,\ldots,d_r\mid m_r}
\frac{d_1\mu(m_1/d_1)\cdots d_r\mu(m_r/d_r)} {\lcm(d_1,\ldots,d_r)}
N_G(d_1,\ldots,d_r).
\end{equation}
\end{theorem}

\begin{proof} Using formula \eqref{Ramanujan_repr} we obtain
\begin{equation*}
E_G(m_1,\ldots,m_r)= \frac1{m} \sum_{k=1}^m \sum_{d_1\mid
\gcd(g_1(k),m_1)} d_1 \mu(m_1/d_1) \ldots \sum_{d_r\mid
\gcd(g_r(k),m_r)} d_r \mu(m_r/d_r)
\end{equation*}
\begin{equation*}
= \frac1{m} \sum_{d_1\mid m_1,\ldots,d_r\mid m_r} d_1\mu(m_1/d_1)
\cdots d_r \mu(m_r/d_r) \sum_{\substack{1\le k\le m
\\ g_1(k)\equiv 0 \text{ (mod $d_1$)},\ldots, g_r(k)\equiv 0 \text{ (mod $d_r$)}}}
1,
\end{equation*}
where the inner sum is $(m/\lcm(d_1,\ldots,d_r))
N_G(d_1,\ldots,d_r)$.
\end{proof}

\begin{cor} If $G$ is any system of
polynomials with integer coefficients, then the function
$(m_1,\ldots,m_r) \mapsto E_G(m_1,\ldots,m_r)$ is integer valued and
multiplicative.
\end{cor}

\begin{proof} By Theorem \ref{theorem_E} and Lemma \ref{lemma_sim_cong}
the function $E_G$ is the convolution of multiplicative functions,
hence it is multiplicative. It is integer valued since each term of
the sum in \eqref{repr_E} is an integer.
\end{proof}

\begin{cor} \label{cor_2} For every $m_1,\ldots,m_r\in \N$ and
$\a=(a_1,\ldots,a_r)\in \Z^r$,
\begin{equation*}
E_{(\a)}(m_1,\ldots,m_r):= \frac1{m} \sum_{k=1}^m
c_{m_1}(k-a_1)\cdots c_{m_r}(k-a_r)
\end{equation*}
\begin{equation*}
=\sum_{d_1\mid m_1,\ldots,d_r\mid m_r} \frac{d_1\mu(m_1/d_1)\cdots
d_r\mu(m_r/d_r)} {\lcm(d_1,\ldots,d_r)}N_{(\a)}(d_1,\ldots,d_r),
\end{equation*}
where
\begin{equation*}
N_{(\a)}(d_1,\ldots,d_r)=
\begin{cases}
1, & \text{if} \  \gcd(d_i,d_j)\mid a_i-a_j \ (1\le i,j\le r),\\ 0,
& \text{otherwise}.
\end{cases}
\end{equation*}
\end{cor}

\begin{proof} Apply Theorem \ref{theorem_E} in the case $g_1(x)=x-a_1,\ldots g_r(x)=x-a_r$, $\a
=(a_1,\ldots,a_r)\in \Z^r$ and Lemma \ref{lemma_cong}.
\end{proof}

\begin{cor} For every $m_1,m_2\in \N$ with $m:=\lcm(m_1,m_2)$ and
every $a_1,a_2\in \Z$,
\begin{equation*}
E_{(a_1,a_2)}(m_1,m_2):= \frac1{m} \sum_{k=1}^m
c_{m_1}(k-a_1)c_{m_r}(k-a_2)
\end{equation*}
\begin{equation} \label{form_cor_3_1}
= \sum_{\substack{d_1\mid m_1,d_2\mid m_2\\
\gcd(d_1,d_2)\mid a_1-a_2}} \gcd(d_1,d_2) \mu(m_1/d_1) \mu(m_2/d_2).
\end{equation}

Furthermore, if $|a_1-a_2|=1$, then
\begin{equation} \label{E_G_case_r_2_eval}
E_{(a_1,a_2)}(m_1,m_2) = \begin{cases} (-1)^{\omega(m)}, & \text{
if}\ m_1=m_2=m \text{ is squarefree},\\ 0, & \text{ otherwise},
\end{cases}
\end{equation}
$\omega(m)$ denoting the number of distinct prime factors of $m$.
\end{cor}

\begin{proof} Formula \eqref{form_cor_3_1} follows from
Corollary \ref{cor_2} applied for $r=2$. Assume that $|a_1-a_2|=1$.
Then \eqref{form_cor_3_1} gives
\begin{equation*}
E_{(a_1,a_2)}(m_1,m_2)= \sum_{\substack{d_1\mid m_1,d_2\mid m_2\\
\gcd(d_1,d_2)=1}} \mu(m_1/d_1) \mu(m_2/d_2),
\end{equation*}
and we deduce that for any prime $p$ and any $u,v\in \N \cup \{0\}$,
\begin{equation*}
E_{(a_1,a_2)}(p^u,p^v)= \begin{cases} 1, & \text{ if}\ u=v=0, \\ -1, & \text{ if}\ u=v=1, \\
0, & \text{ otherwise}.
\end{cases}
\end{equation*}

Since the function $(m_1,m_2)\mapsto E_{(a_1,a_2)}(m_1,m_2)$ is
multiplicative, this leads to \eqref{E_G_case_r_2_eval}.
\end{proof}

For the function $E$ given by \eqref{def_func_E} we deduce from
Corollary \ref{cor_2} the next formula by selecting $a_1=\ldots=a_r=0$.

\begin{cor} {\rm (\cite[Prop.\ 3]{Tot2011})}  For every $m_1,\ldots,m_r\in
\N$,
\begin{equation*}
E(m_1,\ldots m_r)= \sum_{d_1\mid m_1,\ldots,d_r\mid m_r}
\frac{d_1\mu(m_1/d_1)\cdots d_r\mu(m_r/d_r)} {\lcm(d_1,\ldots,d_r)}.
\end{equation*}
\end{cor}

For other special choices of the polynomials $g_1,\ldots,g_r$
similar results can be derived if the values $N_G(d_1,\ldots,d_r)$
are known. We give the following simple example.

\begin{cor} \label{cor_E_G_quadratic}  For every $n\in \N$ write
$n=2^jm$ with $j\in \N \cup \{0\}$ and $m$ odd. Then
\begin{equation} \label{E_G_qudratic}
\frac1{n} \sum_{k=1}^n c_n(k^2-1) = \begin{cases} c_j\ne 0, & \text{
if $j\in \{0,2,3\}$ and $m$ is squarefree},\\ 0, & \text{
otherwise},
\end{cases}
\end{equation}
where $c_0=c_2=1$, $c_3=2$.
\end{cor}

\begin{proof} Apply formula \eqref{repr_E} in the case $r=1$, $g_1(x)=x^2-1$.
For the number $N(p^a)$ of solutions of the congruence $x^2\equiv 1$
(mod $p^a$) it is known (see eg., \cite[Ch.\ 3]{NivZucMont1991}),
that $N(p^a)=2$ ($p$ odd prime, $a\in \N$), $N(2)=1$, $N(4)=2$,
$N(2^{\ell})=4$ ($\ell \ge 3$). Now, \eqref{E_G_qudratic} is
obtained by using the multiplicativity of the involved functions.
\end{proof}

\section{The function $R_G$}

Consider now the integer valued function $R_G$ defined by
\eqref{def_func_R_G}.

\begin{theorem} \label{theorem_R} If $G$ is an arbitrary system of polynomials with
integer coefficients, then for any $m_1,\ldots,m_r\in \N$,
\begin{equation} \label{repr_R}
R_G(m_1,\ldots,m_r) =  \phi(m) \sum_{d_1\mid m_1,\ldots,d_r\mid m_r}
\frac{d_1\mu(m_1/d_1) \cdots d_r\mu(m_r/d_r)}
{\phi(\lcm(d_1,\ldots,d_r))} \eta_G(d_1,\ldots,d_r).
\end{equation}
\end{theorem}

\begin{proof} As in the proof of Theorem \ref{theorem_E}, we write
\begin{equation*}
R_G(m_1,\ldots,m_r)= \sum_{\substack{k=1\\
\gcd(k,m)=1}}^m \sum_{d_1\mid \gcd(g_1(k),m_1)} d_1 \mu(m_1/d_1)
\ldots \sum_{d_r\mid \gcd(g_r(k),m_r)} d_r \mu(m_r/d_r)
\end{equation*}
\begin{equation*}
= \sum_{d_1\mid m_1,\ldots,d_r\mid m_r} d_1\mu(m_1/d_1)
\cdots d_r \mu(m_r/d_r) \sum_{\substack{1\le k\le m, \gcd(k,m)=1, \\
g_1(k)\equiv 0 \text{ (mod $d_1$)},\ldots, g_r(k)\equiv 0 \text{
(mod $d_r$)}}} 1,
\end{equation*}
where the inner sum is $(\phi(m)/\phi(\lcm(d_1,\ldots,d_r)))
\eta_G(d_1,\ldots,d_r)$ by Lemma \ref{lemma_phi}.
\end{proof}

\begin{cor} If $G$ is any system of polynomials with integer coefficients, then the function
$(m_1,\ldots,m_r) \mapsto R_G(m_1,\ldots,m_r)$ is multiplicative. In
particular, in the one variable case (with $r=1$, $g_1=g$,
$\eta_G=\eta_g$ and $m_1=m$) the function of a single variable
\begin{equation} \label{R_g}
R_g(m):= \sum_{\substack{k=1\\
\gcd(k,m)=1}}^m c_m(g(k)) = \phi(m) \sum_{d\mid m}
\frac{d\mu(m/d)}{\phi(d)}\eta_g(d),
\end{equation}
is multiplicative.
\end{cor}

\begin{proof} By Theorem \ref{theorem_R} and Lemma \ref{lemma_sim_cong}
the function $R_G$ is the convolution of multiplicative functions,
hence it is multiplicative.
\end{proof}

\begin{cor} Assume that $g_1=\ldots =g_r=g$ and $m_1,\ldots,m_r\in \N$ are pairwise
relatively prime. Then
\begin{equation*}
R_G(m_1,\ldots,m_r)= R_g(m)
\end{equation*}
\end{cor}

\begin{proof} For any $d_1\mid m_1, \ldots, d_r\mid m_r$ we have
$\eta_G(\lcm(d_1,\ldots,d_r))=\eta_g(d_1\cdots d_r)=
\eta_g(d_1)\cdots \eta_g(d_r)$ and we obtain from \eqref{repr_R}
that
\begin{equation*}
R_G(m_1,\ldots,m_r) = \phi(m_1) \sum_{d_1\mid m_1} \frac{d_1
\mu(m_1/d_1)}{\phi(d_1)}\eta_g(d_1) \cdots \phi(m_r) \sum_{d_r\mid
m_r} \frac{d_r \mu(m_r/d_r)}{\phi(d_r)}\eta_g(d_r)
\end{equation*} \begin{equation*}
=R_g(m_1)\cdots R_g(m_r)=R_g(m),
\end{equation*}
by the multiplicativity of the function \eqref{R_g}.
\end{proof}

\begin{cor} \label{cor_8} For every $m_1,\ldots,m_r\in \N$ and every $\a
=(a_1,\ldots,a_r)\in \Z^r$,
\begin{equation*}
R_{(\a)}(m_1,\ldots,m_r):= \sum_{\substack{k=1\\
\gcd(k,m)=1}}^m c_{m_1}(k-a_1)\cdots c_{m_r}(k-a_r)
\end{equation*}
\begin{equation*}
= \phi(m) \sum_{d_1\mid m_1,\ldots,d_r\mid m_r}
\frac{d_1\mu(m_1/d_1) \cdots d_r\mu(m_r/d_r)}
{\phi(\lcm(d_1,\ldots,d_r))} \eta_{(\a)}(d_1,\ldots,d_r),
\end{equation*}
where
\begin{equation} \label{eta_a}
\eta_{(\a)}(d_1,\ldots,d_r):=
\begin{cases}
1, & \text{if} \ \gcd(d_i,a_i)=1 \ (1\le i\le r), \\ & \text{ and} \
\gcd(d_i,d_j)\mid a_i-a_j \ (1\le i,j\le r),\\
0, & \text{otherwise}.
\end{cases}
\end{equation}
\end{cor}

\begin{proof} This is the special case $g_1(x)=x-a_1,\ldots, g_r(x)=x-a_r$
of Theorem \ref{theorem_R}. We use Lemma \ref{lemma_cong} again,
with the observation that if $x$ is a solution of the simultaneous
congruences $x \equiv a_1$ (mod $d_1$),..., $x \equiv a_r$ (mod
$d_r$), then $\gcd(x,d_1)= \gcd(a_1,d_1)$, ...,
$\gcd(x,d_r)= \gcd(a_r,d_r)$. Hence we obtain for the
values of $\eta_{(\a)}(d_1,\ldots,d_r)$ formula \eqref{eta_a}.
\end{proof}

\begin{cor} If $m_1,\ldots,m_r\in \N$ are pairwise relatively prime, then
\begin{equation} \label{pairw_rel_prime}
R_{(\a)}(m_1,\ldots,m_r) = \mu(m) c_{m_1}(a_1)\cdots c_{m_r}(a_r).
\end{equation}
\end{cor}

\begin{proof} If $m_1,\ldots,m_r$ are pairwise relatively prime and $d_1\mid
m_1,\ldots,d_r\mid m_r$, then $\gcd(d_i,d_j)=1$ for any $i\ne j$ and
we obtain from Corollary \ref{cor_8},
\begin{equation*}
R_{\a}(m_1,\ldots,m_r) = \phi(m_1) \sum_{\substack{d_1\mid m_1\\
\gcd(d_1,a_1)=1}}  \frac{d_1\mu(m_1/d_1)}{\phi(d_1)}
\cdots \phi(m_r) \sum_{\substack{d_r\mid m_r\\
\gcd(d_r,a_r)=1}} \frac{d_r\mu(m_r/d_r)}{\phi(d_r)}
\end{equation*}
\begin{equation*}
= \mu(m_1)c_{m_1}(a_1)\cdots \mu(m_r) c_{m_r}(a_1)=
\mu(m)c_{m_1}(a_1)\cdots c_{m_r}(a_r),
\end{equation*}
by the Brauer-Rademacher identity \eqref{Brauer_Rademacher}.
\end{proof}

Here \eqref{pairw_rel_prime} is an extension of Cohen's formula
\eqref{Cohen_id} for several variables. Note that in the case
$a_1=\ldots =a_r=a$ the right hand side of \eqref{pairw_rel_prime}
is $\mu(m)c_m(a)$.

\begin{cor} If $m_1,m_2\in \N$ with $m=\lcm(m_1,m_2)$ and $a_1,a_2\in \Z$, then
\begin{equation*}
R_{(a_1,a_2)}(m_1,m_2):= \sum_{\substack{k=1\\
\gcd(k,m)=1}}^m c_{m_1}(k-a_1) c_{m_2}(k-a_2)
\end{equation*}
\begin{equation} \label{cor_10}
= \phi(m) \sum_{\substack{d_1\mid m_1, d_2\mid m_2\\
\gcd(d_1,a_1)=1,\gcd(d_2,a_2)=1\\ \gcd(d_1,d_2)\mid a_1-a_2}}
\frac{d_1\mu(m_1/d_1)d_2\mu(m_2/d_2)} {\phi(\lcm(d_1,d_2))}.
\end{equation}

Furthermore, if $\gcd(a_1,m_1)=\gcd(a_2,m_2)=1$ and $|a_1-a_2|=1$,
then
\begin{equation} \label{psi_phi}
R_{(a_1,a_2)}(m_1,m_2) =
\begin{cases} (-1)^{\omega(\gcd(m_1,m_2))} \psi(\gcd(m_1,m_2)), & \text{
if $m_1$ and  $m_2$ are squarefree}, \\
0, & \text{ otherwise}, \end{cases}
\end{equation}
where $\psi(n)=n\prod_{p\mid n} (1+1/p)$ is the Dedekind function.
\end{cor}

\begin{proof} Apply Corollary \ref{cor_8} in the case $r=2$.
Assuming that $\gcd(a_1,m_1)=\gcd(a_2,m_2) =1$ and $|a_1-a_2|=1$ we
deduce from \eqref{cor_10} that for any prime powers $p^u$, $p^v$
($u,v\in \N \cup \{0\}$),
\begin{equation*}
R_{(a_1,a_2)}(p^u,p^v)= \begin{cases} 1, & \text{
if $u=v=0$ or $u=1$, $v=0$ or $u=0$, $v=1$}, \\
-(p+1), & \text{ if $u=v=1$}, \\
0, & \text{ otherwise},
\end{cases}
\end{equation*}
leading to \eqref{psi_phi} by using the multiplicativity of the
function $(m_1,m_2)\mapsto R_{(a_1,a_2)}(m_1,m_2)$.
\end{proof}

We remark that other special systems $G$ can be considered too. As a
further example, we give the following one ($r=1$, $g_1(x)=x^2-1$).
Its proof is similar to that of Corollary \ref{cor_E_G_quadratic}.

\begin{cor} For every $n\in \N$ write
$n=2^jm$ with $j\in \N \cup \{0\}$ and $m$ odd. Then
\begin{equation*} \label{R_G_qudratic}
\sum_{\substack{k=1\\ \gcd(k,n)=1}}^n c_n(k^2-1) =
\begin{cases} d_j \psi(m), & \text{
if $j\in \{0,1,2,3\}$ and $m$ is squarefree},\\ 0, & \text{
otherwise},
\end{cases}
\end{equation*}
where $d_0=d_1=1$, $d_2=4$, $d_3=16$.
\end{cor}

\section{The function $R$}

In this section we investigate the multiplicative function $R$
defined by \eqref{def_func_R}.

\begin{cor} The function $R$ can be represented for arbitrary
$m_1,\ldots,m_r\in \N$ as
\begin{equation} \label{R_eval}
R(m_1,\ldots,m_r)= \phi(m) \sum_{d_1\mid m_1,\ldots,d_r\mid m_r}
\frac{d_1\mu(m_1/d_1) \cdots d_r\mu(m_r/d_r)}
{\phi(\lcm(d_1,\ldots,d_r))}.
\end{equation}
\end{cor}

\begin{proof} Apply Corollary \ref{cor_8} by choosing $a_1=\ldots =a_r=1$.
\end{proof}

The values of $R$ are determined by the next result, which is the
analog of \cite[Lemma 2]{Lis2010}.

\begin{theorem} \label{theorem_R_spec} Let $p^{e_1},\ldots, p^{e_r}$ be any powers of a prime $p$
($e_1,\ldots,e_r\in \N$). Assume, without loss of generality, that
$e:=e_1=e_2=\ldots =e_s>e_{s+1}\ge e_{s+2}\ge \ldots \ge e_r\ge 1$
($r\ge s\ge 1$). Then
\begin{equation*}
R(p^{e_1},\ldots,p^{e_r})= \begin{cases} p^{v+e} (p-1)^{r-s+1}
h_{s}(p), & e>1 \\ (p-1)^r + (-1)^r(p-2), & e=1,
\end{cases}
\end{equation*}
where the integer $v$ is defined by $v=\sum_{j=1}^r e_j-r-e+1$ and
\begin{equation*}
h_s(x)= \frac{(x-1)^{s-1}+(-1)^s}{x}
\end{equation*}
is a polynomial of degree $s-2$ (for $s>1$).
\end{theorem}

\begin{proof} We do not use \eqref{R_eval}, but the definition \eqref{def_func_R} and obtain
with the notation $k=1+dp^{e-1}$,
\begin{equation*}
R(p^{e_1},\ldots,p^{e_r})= \sum_{\substack{0\le d\le p-1\\
\gcd(1+dp^{e-1},p^e)=1}} c_{p^{e_1}}(dp^{e-1}) \cdots
c_{p^{e_r}}(dp^{e-1}),
\end{equation*}
where all the other terms are zero, cf. proof of \cite[Lemma
2]{Lis2010}. If $e>1$, then the condition $\gcd(1+dp^{e-1},p^e)=1$
is valid for every $d\in \{0,1,\ldots,p-1\}$ and obtain
\begin{equation*}
R(p^{e_1},\ldots,p^{e_r})= \sum_{0\le d\le p-1}
c_{p^{e_1}}(dp^{e-1}) \cdots c_{p^{e_r}}(dp^{e-1})= p^e
E(p^{e_1},\ldots,p^{e_r}),
\end{equation*}
and use the corresponding formula for $E(p^{e_1},\ldots,p^{e_r})$.

For $e=1$ one has $r=s$ and
\begin{equation*}
R(p,\ldots,p)= \sum_{0\le d\le p-2} c_{p}(d) \cdots c_{p}(d)=
\left(\sum_{1\le d\le p-2} (-1)^r+ (p-1)^r \right)
\end{equation*}
\begin{equation*}
= (p-2)(-1)^r+ (p-1)^r.
\end{equation*}
\end{proof}

\begin{cor} All values of the function $R$ are nonnegative.
Furthermore, $R(p^{e_1},\ldots,p^{e_r})=0$ if and only if $e>1$ with
$s=1$ or $s$ odd and $p=2$.
\end{cor}

\begin{cor}  For every prime $p$ and every $e_1\ge e_2\ge 1$,
\begin{equation*}
R(p^{e_1},p^{e_2})= \begin{cases} 0, & e_1> e_2 \ge 1,\\
p^{2e-1}(p-1), & e_1=e_2=e>1, \\ p^2-p-1, & e_1=e_2=1.
\end{cases}
\end{equation*}
\end{cor}

\begin{proof} This the case $r=2$ of Theorem \ref{theorem_R_spec}.
\end{proof}

In the case $m_1=\ldots =m_r=m$ let
\begin{equation} \label{def_g_r}
g_r(m)=\frac1{m} R(m,\ldots,m)= \frac1{m} \sum_{\substack{k=1 \\
\gcd(k,m)=1}}^m (c_{m}(k-1))^r.
\end{equation}

The following result is the analog of \cite[Prop.\ 11]{Tot2011}
concerning the function $f_r(m)=\frac1{m}E(m,\ldots,m)$.

\begin{theorem} \label{average_order_g_r}
Let $r\ge 2$. The average order of the function $g_r(m)$ is
$\alpha_r m^{r-1}$, where
\begin{equation*}
\alpha_r:=  \prod_p \left(1+\frac{x_r(p)-p^r}{p^{r+1}}+
\frac{p(p-1)h_r(p)-x_r(p)}{p^{r+2}} \right),
\end{equation*}
$x_r(p)=(p-1)^r+(-1)^r(p-2)$ and $h_r(p)$ is defined in Theorem
\ref{theorem_R}.

More exactly, for any $0< \varepsilon < 1$,
\begin{equation} \label{asymp_g_r}
\sum_{m\le x} g_r(m)= \frac{\alpha_r}{r} x^{r} +{\cal
O}(x^{r-1+\varepsilon}).
\end{equation}
\end{theorem}

\begin{proof} Similar to the proof of \cite[Prop.\ 11]{Tot2011}. The function $g_r$
is multiplicative and by Theorem \ref{theorem_R_spec} and
\eqref{def_g_r}, for every prime $p$,
\begin{equation*}
g_r(p)=\frac{x_r(p)}{p}, \qquad g_r(p^e)=p^{(e-1)(r-1)}(p-1)h_r(p)
\quad (e\ge 2).
\end{equation*}

Therefore,
\begin{equation*}
\sum_{m=1}^{\infty} \frac{g_r(m)}{m^s} = \zeta(s-r+1) \prod_p \left(
1+\frac{a_r(p)}{p^s} + \frac{b_r(p)}{p^{2s}}\right)
\end{equation*}
where
\begin{equation*}
a_r(p)=\frac{x_r(p)}{p}-p^{r-1}, \qquad
b_r(p)=p^{r-1}(p-1)h_r(p)-p^{r-2}x_r(p).
\end{equation*}
for $s\in \C$, $\RE s>r$, the infinite product being absolutely
convergent for $\RE s>r-1$. Consequently, $g_r=F_r* \id_{r-1}$ in
terms of the Dirichlet convolution, where $F_r$ is multiplicative
and for any prime $p$, $F_r(p)=a_r(p)$, $F_r(p^2)=b_r(p)$,
$F_r(p^k)=0$ ($k\ge 3$).

The asymptotic formula \eqref{asymp_g_r} follows by usual estimates.
\end{proof}

\section{The function $T_a$} \label{section_T}

For the function $T_a$ defined by \eqref{def_T_a} we prove the
following results. The first one is a modified orthogonality
relation of the Ramanujan sums.

\begin{theorem} For any $m_1,\ldots,m_r\in \N$ and any $a\in \Z$,
\begin{equation} \label{eval_T}
T_a(m_1,\ldots,m_r)= \begin{cases} m^{r-1}\mu(m)c_m(a), & m_1=\ldots=m_r=m, \\
0, & \text{ otherwise}.
\end{cases}
\end{equation}
\end{theorem}

\begin{proof} We have
\begin{equation*}
T_a(m_1,\ldots,m_r)= \sum_{\substack{j,\ell \text{ (mod $m$)}\\
j+\ell\equiv a \text{ (mod $m$)}\\ \gcd(\ell,m)=1}}
\sum_{\substack{k_1,\ldots,k_{r-1},k_r \text{ (mod $m$)}
\\ k_1+\ldots +k_{r-1}+k_r\equiv j \text{ (mod $m$)} }}
c_{m_1}(k_1)\cdots c_{m_{r-1}}(k_{r-1})c_{m_r}(k_r)
\end{equation*}
\begin{equation*}
= \sum_{\substack{j,\ell \text{ (mod $m$)}\\
j+\ell\equiv a \text{ (mod $m$)}\\ \gcd(\ell,m)=1}}
\left(c_{m_1}(\DOT)\otimes \cdots \otimes c_{m_{r-1}}(\DOT)\otimes
c_{m_r}(\DOT)\right)(j)
\end{equation*}
\begin{equation*}
= \sum_{\substack{\ell \text{ (mod $m$)}\\
\gcd(\ell,m)=1}} \left(c_{m_1}(\DOT)\otimes \cdots \otimes
c_{m_{r-1}}(\DOT)\otimes c_{m_r}(\DOT)\right)(a-\ell)=
\sum_{\substack{\ell \text{ (mod $m$)}\\
\gcd(\ell,m)=1}} K(a-\ell),
\end{equation*}
where $K=c_{m_1}(\DOT)\otimes \cdots \otimes
c_{m_{r-1}}(\DOT)\otimes c_{m_r}(\DOT)$ is the Cauchy convolution
defined by \eqref{Cauchy_convo}.

The functions $k\mapsto c_{m_i}(k)$ are $m_i$-even, hence also
$m$-even, with $m:=\lcm(m_1,\ldots,m_r)$ ($1\le i\le r$). Their
Fourier coefficients are $\alpha_{c_{m_i}}(d)=1$ for $d=m_i$ and $0$
for $d\ne m_i$ ($d\mid m$). We obtain that the function $K$ is
$m$-even and its Fourier coefficients are
\begin{equation*}
\alpha_K(d)= \begin{cases} m^{r-1}, & \text{ if $d=m_1=\ldots
m_r=m$},
\\ 0, & \text{ otherwise}. \end{cases}
\end{equation*}

Now applying Lemma \ref{lemma_4} for the function $K$ we deduce that
\begin{equation*}
T_a(m_1,\ldots,m_r)=\phi(m) \sum_{d\mid m} \frac{\alpha_K(d)
\mu(d)}{\phi(d)} c_d(a)= \begin{cases} m^{r-1}\mu(m)c_m(a), & m_1=\ldots=m_r=m, \\
0, & \text{ otherwise}.
\end{cases}
\end{equation*}
\end{proof}

If $r=1$, then \eqref{eval_T} reduces to Cohen's identity
\eqref{Cohen_id} given in the Introduction.

\begin{cor} For every $a\in \Z$ the function $(m_1,\ldots,m_r)\mapsto T_a(m_1,\ldots,m_r)$ is multiplicative.
\end{cor}

\begin{proof} Using the definition of the multiplicativity, let $m_1,\ldots,m_r,n_1,\ldots,n_r\in \N$ and
assume that $\gcd(m_1n_1,\ldots,m_rn_r)=1$. Let
$m:=\lcm(m_1,\ldots,m_r)$, $n:=\lcm(n_1,\ldots,n_r)$. Then \\
$\lcm(m_1n_1,\ldots,m_rn_r)=mn$, where $\gcd(m,n)=1$ and by
\eqref{eval_T} obtain that
\begin{equation*}
T_a(m_1n_1,\ldots,m_rn_r)= \begin{cases} (mn)^{r-1}\mu(mn)c_{mn}(a), & m_1n_1=\ldots=m_rn_r, \\
0, & \text{ otherwise} \end{cases}
\end{equation*}
\begin{equation*} =
\begin{cases} m^{r-1}\mu(m)c_m(a)n^{r-1}\mu(n)c_n(a), & m_1=\ldots=m_r, n_1=\ldots=n_r, \\
0, & \text{ otherwise}\end{cases}
\end{equation*}
\begin{equation*} =T_a(m_1,\ldots,m_r)T_a(n_1,\ldots,n_r).
\end{equation*}
\end{proof}

In the case $r=2$ the following orthogonality relation holds.

\begin{cor}
For any $m_1,m_2\in \N$ and any $a\in \Z$,
\begin{equation*}
\frac1{m} \sum_{\substack{k,\ell \text{ (mod $m$)} \\
\gcd(\ell,m)=1}} c_{m_1}(k)c_{m_2}(k+\ell-a) = \begin{cases} \mu(m)c_m(a), & m_1=m_2=m, \\
0, & \text{ otherwise}.
\end{cases}
\end{equation*}
\end{cor}

\section{Acknowledgement}
The author thanks the anonymous referee for a very careful reading
of the manuscript and for many helpful suggestions on the
presentation of this paper. The author thanks also Professor
V.~A.~Liskovets for useful remarks.

\end{document}